\documentclass[letterpaper, 10 pt, conference]{ieeeconf}  % Comment this line out if you need a4paper

\IEEEoverridecommandlockouts                              % This command is only needed if 
\usepackage{subcaption}
\usepackage{standalone}    % For including standalone figures
\usepackage{amsmath}    % for align, gather, multiline equations
\usepackage{amssymb}    % extra math symbols
\usepackage{tikz}
\usepackage{tikz-network} % defines \Vertex, \Edge
\usetikzlibrary{calc, positioning, graphs, backgrounds}
\usepackage{algorithm}      % floating environment for algorithms
\newtheorem{definition}{Definition}
\newtheorem{example}{Example} % <-- define the example environment
\usepackage{mathtools} % defines \coloneqq
\usepackage{cite}
\usepackage{todonotes}
\newtheorem{assumption}{Assumption}
\usepackage{algpseudocode}  % Provides \State, \For, etc.
\usepackage{dsfont}
\newtheorem{theorem}{Theorem}
\newtheorem{lemma}{Lemma}
\newtheorem{proposition}{Proposition}

\title{\LARGE \bf
Fiedler-Based Characterization and Identification of \\Leaders in Semi-Autonomous Networks
}

\author{Evyatar Matmon$^{1}$ and Daniel Zelazo$^{2}$% <-this % stops a space
\thanks{*This work was supported by the Israel Science Foundation grant no.
453/24.}% <-this % stops a space
\thanks{$^{1}$Evyatar Matmon is with Technion Autonomous Systems and Robotics Program, Technion - Israel Institute of Technology, Haifa 32000 {\tt\small evyatarmat1@gmail.com}}%
\thanks{$^{2}$Daniel Zelazo with the Stephen B. Klein Faculty of Aerospace Engineering, Technion - Israel Institute of Technology, Haifa 32000         {\tt\small dzelazo@technion.ac.il}}%
}

\begin{document}

\maketitle
\thispagestyle{empty}
\pagestyle{empty}

%%%%%%%%%%%%%%%%%%%%%%%%%%%%%%%%%%%%%%%%%%%%%%%%%%%%%%%%%%%%%%%%%%%%%%%%%%%%%%%%
\begin{abstract}
This paper addresses the problem of identifying leader nodes in semi-autonomous consensus networks from observed agent dynamics. 
Using the grounded Laplacian formulation, we derive graph-theoretic conditions that ensure the components of the Fiedler vector associated with leader and follower nodes are distinct. 
Building on this foundation, we employ the notion of \emph{relative tempo} from prior work as an observable quantity that relates agents’ steady-state velocities to the Fiedler vector. 
This relationship enables the development of a data-driven algorithm that reconstructs the Fiedler vector—and consequently identifies the leader set—using only steady-state velocity measurements, without requiring knowledge of the network topology. 
The proposed approach is validated through numerical examples, demonstrating how graph-theoretic properties and relative tempo measurements can be combined to reveal hidden leadership structures in consensus networks.

\end{abstract}

%%%%%%%%%%%%%%%%%%%%%%%%%%%%%%%%%%%%%%%%%%%%%%%%%%%%%%%%%%%%%%%%%%%%%%%%%%%%%%%%
\section{INTRODUCTION}

The ability to infer or influence the structure of a networked dynamical system is central to understanding and controlling collective behaviors in multi-agent systems. In many practical scenarios, agents interact through local rules without global supervision, and only partial measurements of their states are available. Identifying how such interactions shape the network dynamics has therefore become an important research direction in network systems theory~\cite{Goncalves2008, Shahrampour2013, Hendrickx2018, GEVERS201710580, WEERTS2018247}.

 The \emph{network identification problem} concerns the task of reconstructing or inferring the underlying interaction topology of a multi-agent system from observed data. In other words, given measurements of the agents’ states or outputs, one seeks to determine which agents exchange information and with what strength. This problem is fundamental in areas such as systems biology \cite{Tegner2003}, neuroscience \cite{Friston2011}, social networks \cite{Easley2010}, and engineering systems \cite{Materassi2012}. A central challenge is that the observed collective behavior does not uniquely determine the network structure without additional assumptions or prior knowledge. Accordingly, approaches to network identification are typically organized along a spectrum: at one end, methods assume complete dynamical models and seek to refine unknown network parameters, while at the other end, data-driven methods attempt to infer both the dynamics and the topology simultaneously \cite{Timme2014,Ren2007}. In the context of multi-agent coordination, identifying the interaction graph is crucial for understanding how local rules give rise to global behaviors, as well as for designing control strategies that are robust to uncertainties or malicious attacks in the communication topology. An illustration of network topology identification is presented in Figure~\ref{fig:NetworkIdIllus}.

\begin{figure}[!htbp] 
  \centering
  %\includestandalone[mode=buildnew,width=\linewidth]{graphics/NetworkIdIllus}
   \includegraphics[width=0.9\linewidth]{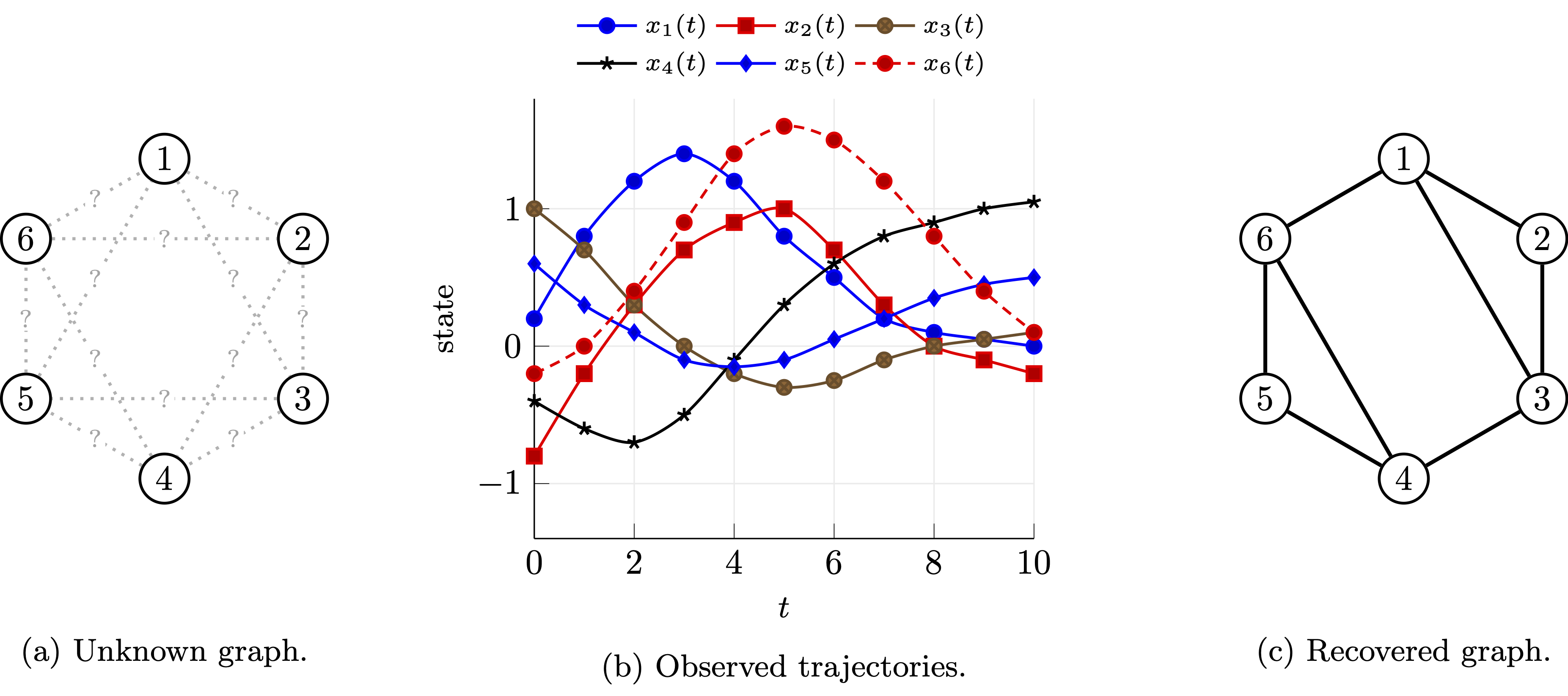}
  \caption{Conceptual overview of the network identification problem: (a) the interaction graph is unknown; (b) trajectories are observed;
  (c) after identification, the edges are recovered.}
  \label{fig:NetworkIdIllus}
\end{figure}

Early work focused on offline reconstruction of the interaction topology or coupling weights from state or input--output data, often assuming access to global trajectories~\cite{Materassi2012}. More recent efforts have sought weaker conditions for identifiability, studying local structural properties, graph-theoretic methods, or algebraic constraints that ensure unique reconstruction up to graph symmetries~\cite{VANWAARDE2018319,8382192}. Other works approach the problem by probing the network with specially designed inputs, \cite{Timme2007, NetworkIdentificationDiffusivelyCoupledNetworksMinimalTimeComplexity}, or by node knock-out \cite{Nabi2010}.

In many real-world networks, not all agents behave identically—some may possess additional autonomy or receive external commands.  These \emph{semi-autonomous consensus networks} consist of a subset of agents, often termed \emph{leaders} or \emph{externally driven nodes}, that are subject to independent control inputs or reference signals, while the remaining \emph{followers} update their states through local interactions within the network~\cite{NI2010209, REN2007474, Li01012019}.  Such leader–follower architectures appear in diverse applications including vehicle platooning, robotic swarms, and opinion dynamics, where overall coordination must be preserved despite partial autonomy.  

From a systems-theoretic viewpoint, these networks introduce an additional layer of structure atop the standard consensus model: the interconnection topology not only determines how information flows, but also which nodes actively influence collective behavior.  Consequently, an important question is how to \emph{identify the leaders}—that is, determine which agents receive external inputs or exert disproportionate influence on the group—using only observed data.  This \emph{leader identification problem} can be regarded as a focused instance of the broader \emph{network identification problem}, where the goal is to infer the underlying interaction structure or influence pattern from measurements of the agents’ trajectories~\cite{Dai2019}.  To the best of our knowledge, this specific problem has received little direct attention in the literature.

% \textcolor{blue}{In semi-autonomous consensus networks, a subset of agents—often called \emph{leaders} or \emph{externally driven nodes}—receive additional direct control inputs or follow independent reference trajectories, while the remaining \emph{followers} update their states based on within network interactions. Such architectures arise in applications ranging from coordinated vehicle platoons to robotic swarms and social influence models, where global coherence must be maintained despite partial autonomy.} Within the context of \emph{leader--follower} or \emph{semi-autonomous} networks, attention has turned to identifying which agents exert influence on the group and how this influence propagates through the consensus dynamics. %\todo[inline]{need more introduction - a little on semi-autonomous consensus, and also on relative tempo \textcolor{blue}{there is a relative tempo intro in the main contribution paragraph}}

The main contribution of this paper is the development of a data-driven algorithm for identifying leader nodes in semi-autonomous consensus networks. 
Building on a new spectral characterization of leader–follower graphs, we first establish analytical conditions under which the components of the Fiedler vector corresponding to leaders and followers are separable. 
We then leverage the concept of \emph{relative tempo} from~\cite{Shao2019relative} as a measurable proxy for the Fiedler vector, enabling estimation of the network’s influence structure directly from steady-state velocity data. 
This connection leads to a simple yet effective leader identification procedure that does not require knowledge of the underlying topology. 
The proposed framework therefore unifies  graph analysis with data-driven inference and provides new insights into how leader influence manifests in the dynamical response of the network.

The rest of the paper is organized as follows.  Section \ref{sec.statement} introduces the general problem setup and the semi-autonomous consensus protocol.  Section~\ref{sec.leader_graphs} establishes sufficient graph-theoretic conditions for leader identifiability and characterizes the limiting form of the Fiedler vector. 
Section~\ref{sec:algorithm} reviews the relative tempo formulation and its connection to the Fiedler eigenvector, culminating in a data-driven leader identification algorithm. 
Section~\ref{sec:NumericalExamples} presents simulation results that validate the theoretical findings, and Section~\ref{sec.conclusion} concludes with remarks and directions for future work.

\paragraph*{Notations}
Let $\mathbb{R}$ and $\mathbb{N}$ denote the set of real numbers and natural numbers, respectively, and $\mathbb{R}^{m\times n}$ the space of real matrices of size $m\times n$. The $n$-dimensional vector of all ones is denoted $\mathds{1}_n$, and $I_n$ is the identity matrix. A graph $\mathcal G=(\mathcal{V},\mathcal{E},\mathcal W)$ is defined by a set of vertices $\mathcal{V} = \{v_1,\dots,v_n\}$, edges $\mathcal E \subseteq \mathcal V \times \mathcal V$, and weights $\mathcal W:\mathcal E \to \mathbb R$. For directed edges, $e_{ij}=(v_i,v_j) \in \mathcal E$ denotes a directed edge from $v_i$ to $v_j$, while for undirected edges $e_{ij} = \{v_i,v_j\}$. The notation $i \sim j$ indicates that an edge exists between the nodes $i$ and $j$, while $\mathcal N(i) = \{j\in\mathcal V \,|\, i \sim j\}$ is the \emph{neighbor set} of node $i$. The \emph{degree} of node $i$, $\deg(i)$, is the cardinality of its neighbor set.  %A \emph{walk} $\Gamma_{ij}$ is an alternating sequence of vertices and edges from $v_i$ to $v_j$. A \emph{closed walk} has the same start and end node. The length of a walk is the number of edges, denoted $\|\Gamma_{ij}\|$, and its weight is $W_{\Gamma_{ij}} = \prod_{e_{ab} \in \Gamma_{ij}} w_{ab}$.

\section{PROBLEM STATEMENT}\label{sec.statement}

We consider a network of $n$ agents that interact with each other over an undirected and connected graph $\mathcal G$. A subset of the agents, termed the \emph{leaders}, are able to receive additional external commands that may be used to steer the network along a desired trajectory or configuration.  In this direction, let $\mathcal V_\ell \subset \mathcal V$ denote the set of leaders, and $V_f = \mathcal V \setminus \mathcal V_{\ell}$ the set of \emph{followers}.  The dynamics of each agent, therefore, are given by
\begin{align}
    \dot x_i = 
    \begin{cases}
        u_i + u_i^{\mathrm{ex}}, & i \in \mathcal{V}_\ell,\\
        u_i, & i \in \mathcal{V}_f,
    \end{cases}
\end{align}
where $x_i \in \mathbb{R}$ is the state of agent $i$, $u_i \in \mathbb{R}$ is the control, and $u_i^{\mathrm{ex}} \in \mathbb{R}$ is the additional external input available to the leader agents.

We now introduce a basic assumption about the number of leaders and followers. It states that there must be at least one follower node. %, and no leader node should be connected to all the other nodes.\textcolor{blue}{the second condition is not needed. we only care that there is at least one follower.}
\begin{assumption}\label{assump.leaders}
    The follower set is non-empty ($\mathcal V_f \neq \emptyset$).  %Furthermore, $\deg(i) < |\mathcal V|-1$ for all $v_i \in \mathcal V_\ell$.
\end{assumption} %\textcolor{blue}{second condition not needed}

Each agent implements the celebrated \emph{consensus protocol} \cite{Mesbahi2010}, modified such that the leader agents receive the additional input $u_i^{ex}$. That is, each agent has the closed-loop dynamics
\begin{align}\label{semi_aut_orig}
    \dot x_i = 
    \begin{cases}
        \displaystyle \sum_{j \sim i} (x_j - x_i) + (u_i^{\mathrm{ex}} - x_i), & i \in \mathcal{V}_\ell,\\
       \displaystyle \sum_{j \sim i} (x_j - x_i), & i \in \mathcal{V}_f,
    \end{cases}
\end{align}
where the sum is taken over all neighbors $j$ of agent $i$.

In this work we assume that external inputs $u_i^{ex}$ are \emph{constant signals}.  In this way, it is convenient to model them as additional nodes/state in the network, labeled as $y_i \in \mathbb{R}$, with $\dot y_i = 0$ and $y_i(0) = u_i^{ex}$. Defining the augmented state $\bar{x} = [x^T \; y^T]^T$, the closed-loop dynamics can be expressed as 
\begin{equation}\label{semi-aut_aug}
\begin{bmatrix}\dot{x}\\ \dot{y}\end{bmatrix} = -\begin{bmatrix}\bar{L}_{11} & \bar{L}_{12} \\ 0 & 0\end{bmatrix}\begin{bmatrix}x\\y\end{bmatrix},
\end{equation}
with
\begin{equation}\label{L11}
\bar{L}_{11} = L(\mathcal{G}) + \begin{bmatrix} I_{|\mathcal{V}_\ell|} & 0\\ 0 & 0\end{bmatrix}, \quad
\bar{L}_{12} = \begin{bmatrix}-I_{|\mathcal{V}_\ell|}\\0\end{bmatrix}.
\end{equation}
Here, $L(\mathcal G)$ is the graph Laplacian matrix of $\mathcal G$. Note that the state matrix in \eqref{semi-aut_aug} can be thought of as the Laplacian of a \emph{directed graph} $\mathcal{\bar{G}}$, where the external input nodes (the $y_i$'s) are connected by a directed edge to the corresponding leader node.  This is illustrated in Fig. \ref{fig.auggraph}, where the leader nodes (marked in green) have directed edges from the external signals (marked as red nodes).  The remaining edges are bidirectional.

\begin{figure}[!h]
  \centering
%\includestandalone[width=0.35\linewidth]{graphics/AugGraph}
\includegraphics{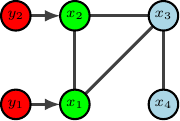}
  \caption{Example of augmented graph $\mathcal{\bar{G}}$. Leaders (green) receive inputs from external sources (red).}\label{fig.auggraph}
\end{figure}

We now observe that the matrix $\bar L_{11}$ in \eqref{semi-aut_aug} is the \emph{grounded Laplacian} of the undirected version of $\mathcal{\bar G}$ \cite{OnTheSmallestEigenvectorGroundedLaplac}, where the external control nodes (the $y_i$'s) are removed.  The \emph{Fiedler eigenvalue}, $\lambda_F$ corresponds to the smallest eigenvalue of $\bar L_{11}$; its corresponding eivenvector $v_F$ can be chosen to have all positive entries. It was shown in \cite{OnTheSmallestEigenvectorGroundedLaplac} that under Assumption \ref{assump.leaders} that $0<\lambda_F<1$. For the remainder of this work, we refer to \emph{the Fiedler eigenvector of $\mathcal G$} to mean the Fiedler vector of the grounded Laplacian $\bar L_{11}$.

In this work we are interested in two problems:
\begin{itemize}
    \item[i)] Characterize a class of graphs such that it is possible to identify the leader nodes from the trajectories of \eqref{semi_aut_orig};
    \item[ii)] Using the trajectory data of \eqref{semi_aut_orig}, construct an algorithm to identify the leader nodes.
\end{itemize}
In the sequel we present our solution to the above problem.

% Considering the symmetric graph $\bar{\mathcal{G}}'$ with Laplacian
% \(
% L(\bar{\mathcal{G}}') = \begin{bmatrix}\bar{L}_{11} & \bar{L}_{12} \\ \bar{L}_{12}^T & \mathrm{diag}(\bar{L}_{12}^T \mathbf{1})\end{bmatrix},
% \)
% the submatrix $\bar{L}_{11}$ corresponds to the grounded Laplacian $L_B$ obtained by removing the $y$ nodes:
% \begin{equation}\label{LBi}
% [L_B]_{ij} = 
% \begin{cases}
% d_i+1, & i=j\in \mathcal{V}_\ell,\\
% d_i, & i=j\in \mathcal{V}_f,\\
% -1, & i\sim j,\\
% 0, & \text{otherwise}.
% \end{cases}
% \end{equation}

% \begin{proposition}[Properties of $L_B$]\label{FiedlerLBproperties}
% For the setup above, the Fiedler eigenvector $v_F$ of $L_B$ is unique up to scaling with all positive entries, and the Fiedler eigenvalue $\lambda_F$ is simple, positive, and satisfies $0<\lambda_F\le 1$.
% \end{proposition}
% \begin{proof}
% $L_B$ results from removing rows/columns corresponding to $y$ from the Laplacian of the augmented connected graph. Lemma~\ref{FiedlerUpperBound0} then implies the stated properties.
% \end{proof}

\section{Leader-Identifiable Graphs}\label{sec.leader_graphs}

It is well-known in the literature that the Fiedler eigenvector can be used to identify clusters in a graph \cite{Ng_NIPS21, Luxburg2007}. In this section we derive conditions that ensure the elements of the Fiedler vector can be separated into two groups, one corresponding to the follower nodes and one to the leader nodes.

We now explore algebraic and combinatorial connections between the Fiedler vector associated with $\bar L_{11}$, and the structure of the graph $\mathcal G$.  %To begin, we first observe that the entries of the Fiedler vector can be related to walks on an associated normalized graphs \textcolor{blue}{I dont think we should related it to walk on graph since we dont use it anywhere. In the thesis I used the walks on graph to develop further expressions but here we dont use it. The Ahat matrix is defined in order to examine the Fiedler vector convergence later.. If we just try to examine the multiplication of $\hat{v}_F$ with the grounded Laplacian its hard to say something about convergence of the Fiedler vector}.  
In this direction, we define the \emph{semi-normalized adjacency matrix} of $\mathcal G$ as
\begin{align}\label{adj_norm}
\hat A(\mathcal G) = \hat D_{\lambda_F}^{-1} A(\mathcal G) \in \mathbb{R}^{n \times n},
\end{align}
where $A(\mathcal G)$ is the adjacency matrix of $\mathcal G$ and $\hat{D}_{{\lambda_F}}\in \mathbb{R}^{n \times n}$ is given by
 \begin{equation}\label{seminorm_degree}
          [\hat{D}_{\lambda_F}]_{ii} = \begin{cases} \deg(i) + 1 - \lambda_F & i\in\mathcal{V}_\ell\\
          \deg(i) - \lambda_F & i\in\mathcal{V}_f
          \end{cases}.
 \end{equation}
 Note that under Assumption \ref{assump.leaders}, $\lambda_F<1$ so $\hat{D}_{\lambda_F}$ is invertible.  We observe that $\hat A(\mathcal G)$ is actually the adjacency matrix of a weighted and directed graph, which we denote as $\hat G$. It is readily verified that $\hat G$ must be strongly connected if and only if $\mathcal G$ is connected.

 The matrix $\hat{A}(\mathcal G)$ is diagonalizable since it is similar to the symmetric matrix $\hat{D}_{\lambda_F}^{-1/2}A\hat{D}_{\lambda_F}^{-1/2}$, as
\begin{equation}\label{AhatSim}
    \hat{A}(\mathcal G) = \hat{D}_{\lambda_F}^{-1/2}(\hat{D}_{\lambda_F}^{-1/2}A(\mathcal G)\hat{D}_{\lambda_F}^{-1/2}) \hat{D}_{\lambda_F}^{1/2}.
\end{equation}
Consequently, $A(\mathcal G)$ is similar to $\hat{A}(\mathcal G)$.

\begin{proposition}\label{AhatEigenvalue}
%\textcolor{blue}{above similarity is also not needed since we dont use it here. In the thesis I use it in Proposition 3.2.7, but here we dont need it.}
The semi-normalized adjacency matrix $\hat{A}(\mathcal G)$ has real, simple, greatest eigenvalue equal to $1$. Furthermore, the corresponding eigenvector is $v_F$,  the  Fiedler vector of $\bar L_{11}$.   
\end{proposition}
\begin{proof}
    Define the diagonal matrix
\[
[\hat D_0]_{ii} =
\begin{cases}
\deg(i)+1, & i\in\mathcal V_\ell,\\
\deg(i), & i\in\mathcal V_f.
\end{cases}
\]
From the definition of the grounded Laplacian,
\[
\bar L_{11} = \hat D_0 - A(\mathcal G).
\]
Subtracting $\lambda_F I$ and left-multiplying by $(\hat D_0-\lambda_F I)^{-1}$ gives
\[
(\hat D_0-\lambda_F I)^{-1}A(\mathcal G)
= I-(\hat D_0-\lambda_F I)^{-1}(\bar L_{11}-\lambda_F I).
\]
Since $\hat D_{\lambda_F}=\hat D_0-\lambda_F I$, we have
\[
\hat A(\mathcal G) = I - \hat D_{\lambda_F}^{-1}(\bar L_{11}-\lambda_F I).
\]
Multiplying by $v_F$ and using $\bar L_{11}v_F=\lambda_F v_F$ yields $\hat A(\mathcal G) v_F=v_F$,
so $\lambda=1$ is an eigenvalue of $\hat A(\mathcal G)$ with eigenvector~$v_F$.

Because $\hat A$ is a nonnegative matrix whose associated graph is strongly
connected, the
Perron–Frobenius theorem
\cite{horn13}
implies that its spectral radius is real, positive, and simple,
and that the corresponding eigenvector can be chosen strictly positive.
Since $v_F$ has positive entries,
the pair $(\lambda,v_F)=(1,v_F)$
is the unique Perron–Frobenius eigenpair of~$\hat A(\mathcal G)$.
\end{proof}

To study how network structure affects the spectral properties of the grounded Laplacian, we introduce the notion of a \emph{graph sequence}. 
The key idea is to construct a family of graphs in which the set of leader nodes remains fixed, while the interconnections among the follower nodes become progressively denser. 
This construction provides a systematic way to examine how the components of the Fiedler vector evolve as network connectivity increases. 
In particular, we will show that along such a sequence, the Fiedler vector converges to a simple, degree-dependent expression. % that admits an analytical approximation.

\begin{definition}[Graph Sequence]\label{def:graph_sequence}
Let $\sigma : \mathbb{N} \to \mathbb{N}$ be an indexing map, and let 
$\mathcal{G}^{\sigma(i)} = (\mathcal{V}^{\sigma(i)}, \mathcal{E}^{\sigma(i)})$ 
denote a sequence of undirected graphs, where the node set is partitioned into disjoint leader and follower subsets,
\[
\mathcal{V}^{\sigma(i)} = 
\mathcal{V}_\ell^{\sigma(i)} \cup \mathcal{V}_f^{\sigma(i)}, 
\text{ and }
\mathcal{V}_\ell^{\sigma(i)} \cap \mathcal{V}_f^{\sigma(i)} = \emptyset.
\]
For each follower node $j \in \mathcal{V}_f^{\sigma(i)}$, define its number of follower and leader neighbors as
\begin{align*}
\deg_f^{\sigma(i)}(j) &= 
|\{ k \in \mathcal{V}_f^{\sigma(i)} \mid (k,j)\in\mathcal{E}^{\sigma(i)}\}|,\\
\deg_\ell^{\sigma(i)}(j) &= 
|\{ k \in \mathcal{V}_\ell^{\sigma(i)} \mid (k,j)\in\mathcal{E}^{\sigma(i)}\}|.
\end{align*}
Furthermore, let $\underline{\deg}_f^{\sigma(i)} = 
\min_{j \in \mathcal{V}_f^{\sigma(i)}} \deg_f^{\sigma(i)}(j)$ denote the minimum follower-follower degree %\textcolor{blue}{should be minimum follower-follower degree, as we denoted in theorem\ref{thm:main} and definition \ref{def:graph_sequence}}.  
The sequence $\{\mathcal{G}^{\sigma(i)}\}$ satisfies the following properties:
\setlength{\itemsep}{0pt}
\setlength{\parskip}{0pt}
\setlength{\parsep}{0pt}
\begin{itemize}
        \item[$i)$] $\mathcal{G}^{\sigma(i)}$ is connected for all $i$;
         \item[$ii)$] The leader set is fixed, 
    $\mathcal{V}_\ell^{\sigma(i)} = \mathcal{V}_\ell^{\sigma(i+1)} = \mathcal{V}_\ell$, 
    with $|\mathcal{V}_\ell| \ge 1$;
        \item[$iii)$]Each leader has constant degree, $\mathbf d_j := \deg_\ell^{\sigma(i)}(j)=\deg_\ell^{\sigma(i+1)}(j)$ for all $i$ and for all $j\in \mathcal{V}_\ell$;
        \item[$iv)$]  Leaders are not directly connected to one another, 
    i.e., $k \notin \mathcal{N}^{\sigma(i)}(j)$ for all $k,j \in \mathcal{V}_\ell$;
        \item[$v)$]The minimum follower degree increases monotonically, 
    $\underline{\deg}_f^{\sigma(i)} < \underline{\deg}_f^{\sigma(i+1)}$ for all $i$.
    \end{itemize}

\end{definition}

%An example of such a sequence is shown in Fig.~\ref{fig:graph_sequence_example}, where follower nodes (blue) are gradually added to a network with constant leader set (green).
\begin{example}\label{GraphSequqnceExample}
Figure~\ref{fig:graph_sequence_example} illustrates a sequence of graphs satisfying
Definition~\ref{def:graph_sequence}.  
The leader set (green nodes) remains fixed throughout, and leaders are never directly connected.  
In contrast, the follower subgraph (blue nodes) becomes progressively denser as new edges and nodes are added.

\end{example}

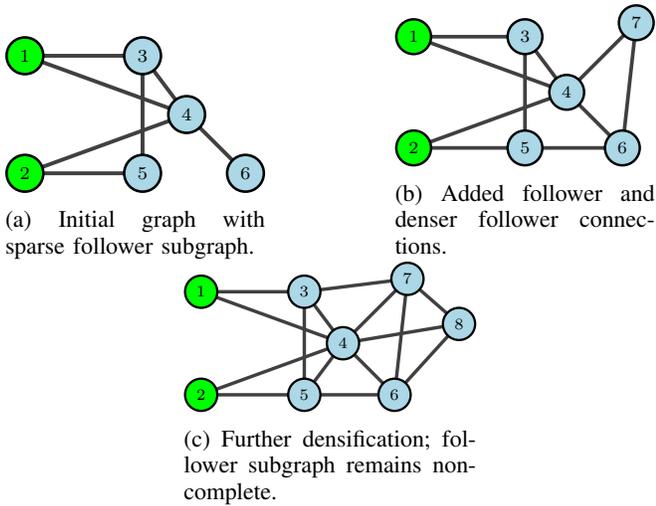
\begin{figure}[!h]
  \centering
  \begin{subfigure}{0.4\linewidth}
    \centering
    \vspace{10pt}\begin{tikzpicture}
\Vertex[size=.5, x=-1, y=1.6, label=$1$, color=green]{L1}
\Vertex[size=.5, x=-1, y=0.0, label=$2$, color=green]{L2}
\Vertex[size=.5, x=0.6,  y=1.6, label=$3$]{F3}
\Vertex[size=.5, x=1.2,  y=0.8, label=$4$]{F4}
\Vertex[size=.5, x=0.6,  y=0.0, label=$5$]{F5}
\Vertex[size=.5, x=2,  y=0.0, label=$6$]{F6}
\Edge(L1)(F3)
\Edge(L1)(F4)
\Edge(L2)(F5)
\Edge(L2)(F4)
\Edge(F3)(F4)
\Edge(F3)(F5)
\Edge(F4)(F6)
\end{tikzpicture}
    \caption{Initial graph with sparse follower subgraph.}
  \end{subfigure}\hfill
  \begin{subfigure}{0.4\linewidth}
    \centering
    \vspace{10pt}\begin{tikzpicture}

\Vertex[size=.5, x=-1, y=1.6, label=$1$, color=green]{L1}
\Vertex[size=.5, x=-1, y=0.0, label=$2$, color=green]{L2}
\Vertex[size=.5, x=0.6,  y=1.6, label=$3$]{F3}
\Vertex[size=.5, x=1.2,  y=0.8, label=$4$]{F4}
\Vertex[size=.5, x=0.6,  y=0.0, label=$5$]{F5}
\Vertex[size=.5, x=2,  y=0.0, label=$6$]{F6}
\Vertex[size=.5, x=2.2,  y=1.8, label=$7$]{F7}
\Edge(L1)(F3)
\Edge(L1)(F4)
\Edge(L2)(F5)
\Edge(L2)(F4)
\Edge(F3)(F4)
\Edge(F4)(F6)
\Edge(F3)(F5)
\Edge(F4)(F7)
\Edge(F6)(F7)
\Edge(F5)(F6)
\end{tikzpicture}
    \caption{Added follower and denser follower connections.}
  \end{subfigure}\hfill
  \begin{subfigure}{0.45\linewidth}
    \centering
    \begin{tikzpicture}

\Vertex[size=.5, x=-1, y=1.6, label=$1$, color=green]{L1}
\Vertex[size=.5, x=-1, y=0.0, label=$2$, color=green]{L2}
\Vertex[size=.5, x=0.6,  y=1.6, label=$3$]{F3}
\Vertex[size=.5, x=1.2,  y=0.8, label=$4$]{F4}
\Vertex[size=.5, x=0.6,  y=0.0, label=$5$]{F5}
\Vertex[size=.5, x=2,  y=0.0, label=$6$]{F6}
\Vertex[size=.5, x=2.2,  y=1.8, label=$7$]{F7}
\Vertex[size=.5, x=3,  y=1.1, label=$8$]{F8}
\Edge(L1)(F3)
\Edge(L1)(F4)
\Edge(L2)(F5)
\Edge(L2)(F4)
\Edge(F3)(F4)
\Edge(F4)(F6)
\Edge(F3)(F5)
\Edge(F4)(F7)
\Edge(F6)(F7)
\Edge(F3)(F7)
\Edge(F7)(F8)
\Edge(F4)(F8)
\Edge(F6)(F8)
\Edge(F5)(F6)
\Edge(F5)(F4)
\end{tikzpicture}
    \caption{Further densification; follower subgraph remains non-complete.}
  \end{subfigure}
  \caption{Example of a graph sequence where leaders (green) remain fixed and followers (blue) become progressively more connected. 
  The leader degrees remain constant across all graphs, and leaders are never directly connected.}
  \label{fig:graph_sequence_example}
\end{figure}

We now present a result that shows, in the limit of the graph sequence defined in Definition \ref{def:graph_sequence}, the structure of the Fiedler vector.
\begin{lemma}\label{FiedlerDegree}
Let $\{\mathcal{G}^{\sigma(i)}\}$ be a graph sequence satisfying Definition~\ref{def:graph_sequence}. Then,
%\begin{itemize}
  %  \item[a)] $\displaystyle \lim_{i\to\infty} \underline{\deg}_F^{\sigma(i)} = \infty$;
   % \item[b)]
    $ \displaystyle\lim_{i\to\infty} [\hat A^{\sigma(i)}(\mathcal G)\bar v_F^{\sigma(i)}]_j = \lim_{i\to\infty} \bar v_F^{\sigma(i)}$, where
    \begin{align}\label{fiedler_vec_limit}
    [ \bar v_F^{\sigma(i)}]_j =
    \begin{cases}
        1, & j\in \mathcal{V}_f^{\sigma(i)},\\[2mm]
        \dfrac{\mathbf d_j}{\mathbf d_j+1- \lambda_F^{\sigma(i)}}, & j\in\mathcal{V}_\ell,
    \end{cases}
    \end{align}
    where $\lambda_F^{\sigma(i)}$ is the Fiedler value of the grounded Laplacian associated with the graph $\mathcal G^{\sigma(i)}$ and leader set $\mathcal V_\ell$.
%\end{itemize}
\end{lemma}

\begin{proof}
    For notational brevity, set $\lambda_i := \lambda_F^{\sigma(i)}$, 
$\hat A_i := \hat A^{\sigma(i)}(\mathcal G)$, 
$\bar v_i := \bar v_F^{\sigma(i)}$, and 
$\deg_i(j) := \deg^{\sigma(i)}(j)$. For any node $j \in \mathcal{V}$ we have
\begin{align}\label{eq:hatAvi_general}
    [\hat A_i \bar v_i]_j
    = 
    \frac{1}{\deg_i(j) + \delta_j - \lambda_i}
    \sum_{k\in \mathcal N_i(j)} [\bar v_i]_k,
\end{align}
where $\delta_j = 1$ if $j\in\mathcal V_\ell$ and $\delta_j=0$ otherwise. Using~\eqref{fiedler_vec_limit} and splitting by follower/leader neighbors,
\begin{align*}
    \sum_{k\in \mathcal N^{\sigma(i)}(j)} [\bar v_i]_k
    &= \hspace{-10pt}
    \sum_{k\in \mathcal N^{\sigma(i)}(j)\cap\mathcal V_f} 1
    +
    \sum_{k\in \mathcal N^{\sigma(i)}(j)\cap\mathcal V_\ell}
    \frac{\mathbf d_k}{\mathbf d_k+1-\lambda_i}\\[1ex]
    &\hspace{-.5cm}=
    \deg_i(j)
    - (1-\lambda_i)
      \sum_{k\in \mathcal N^{\sigma(i)}(j)\cap\mathcal V_\ell}
      \frac{1}{\mathbf d_k+1-\lambda_i},
\end{align*}
where we used the identity $\frac{\mathbf d_k}{\mathbf d_k+1-\lambda_i}
= 1 - \frac{1-\lambda_i}{\mathbf d_k+1-\lambda_i}$.

Substituting into~\eqref{eq:hatAvi_general} and separating the terms yields compact form
\begin{align}\label{eq:hatAvi_closed}
    [\hat A_i \bar v_i]_j
    &=
    1
    -
    \frac{\delta_j-\lambda_i}{\deg_i(j)+\delta_j-\lambda_i}
    - \nonumber\\
    &
    \frac{(1-\lambda_i)}
         {\deg_i(j)+\delta_j-\lambda_i}
      \sum_{k\in \mathcal N^{\sigma(i)}(j)\cap\mathcal V_\ell}
      \frac{1}{\mathbf d_k+1-\lambda_i}.
\end{align}

We now consider \eqref{eq:hatAvi_closed} for $j\in\mathcal V_f$. Here $\delta_j=0$ and $[\bar v_i]_j = 1$.  From~\eqref{eq:hatAvi_closed},
\begin{align*}
[\hat A_i \bar v_i]_j
&=1+ \frac{\lambda_i}{\deg_i(j)-\lambda_i} \\
&- \frac{(1-\lambda_i)}
         {\deg_i(j)-\lambda_i}
      \sum_{k\in \mathcal N^{\sigma(i)}(j)\cap\mathcal V_\ell}
      \frac{1}{\mathbf d_k+1-\lambda_i}.
\end{align*}
By Definition~\ref{def:graph_sequence}($v$),
$\underline{\deg}_f^{\sigma(i)}\to\infty$, 
while $\lambda_i\in(0,1)$, and the leader degrees $\mathbf d_k$ are constant by item~($iii$). 
Hence both correction terms vanish, and therefore
\[
\lim_{i\to\infty} [\hat A_i \bar v_i]_j = 1
= \lim_{i\to\infty} [\bar v_i]_j.
\]
For leader nodes $j \in \mathcal V_\ell$, by Definition~\ref{def:graph_sequence}($iv$), leaders are not directly connected,
so $\mathcal N^{\sigma(i)}(j)\cap\mathcal V_\ell=\emptyset$.
Equation~\eqref{eq:hatAvi_closed} then reduces to
\begin{align*}
[\hat A_i \bar v_i]_j
&=
1 - \frac{1-\lambda_i}{\deg_i(j)+1-\lambda_i}
=
\frac{\deg_i(j)}{\deg_i(j)+1-\lambda_i} \\
&=
\frac{\mathbf d_j}{\mathbf d_j+1-\lambda_i}
= [\bar v_i]_j.
\end{align*}
Hence equality holds for all~$i$. To conclude, we have in both cases \(
\displaystyle\lim_{i\to\infty} [\hat A_i \bar v_i]_j
= \displaystyle\lim_{i\to\infty} [\bar v_i]_j,
\).
\end{proof}

An important consequence of Lemma \ref{FiedlerDegree}, is that in the limit (i.e., for an infinite graph satisfying the properties in Definition \ref{def:graph_sequence}, \eqref{fiedler_vec_limit} is the Fiedler eigenvector.
% \begin{corollary}\label{FiedlerConvergence}
%     For the same conditions stated in Lemma \ref{FiedlerDegree}, the Fiedler vector of $\mathcal{G}^{\sigma(i)}$ converges to the following values: 
%     \begin{equation*}
%      \lim_{i\to\infty}[v_F^{\sigma(i)}]_j = c\lim_{i\to\infty}  \begin{cases}
%             1,& j\in \mathcal{V}_f^{\sigma(i)} \\
%             \frac{\mathbf d_j}{\mathbf d_j+1-\lambda_F^{\sigma(i)}},& j\in \mathcal{V}_\ell
%         \end{cases}, \hspace{3mm}c\in \mathbbm{R}.
% \end{equation*}
% \end{corollary}
% \begin{proof}
% From Proposition \ref{AhatEigenvalue} we know that $\hat A(\mathcal{G})$ has eigenvalue $\lambda =1$ with the corresponding eigenvector $v_F$ which is the Fidler vector. Combined this with Lemma \ref{FiedlerDegree} gives our results.
% \end{proof}

The following theorem leverages our previous results to provide a sufficient condition on the graph structure that ensures a separation between the Fiedler vector components corresponding to leaders and followers. This separation will later enable us to identify the leaders.

\begin{theorem}\label{thm:main}
Let $\mathcal{G} = (\mathcal{V},\mathcal{E})$ be an undirected graph whose node set 
is partitioned into followers $\mathcal{V}_f$ and leaders $\mathcal{V}_\ell$, 
and let $A(\mathcal{G})$ and $\lambda_F$ denote, respectively, 
the adjacency matrix and the Fiedler eigenvalue of the (grounded) Laplacian associated with $\mathcal{G}$.
For each follower node $j \in \mathcal{V}_f$, define
\[
\deg_f(j) = 
\bigl|\{\,k \in \mathcal{V}_f \mid (k,j)\in\mathcal{E}\,\}\bigr|,
\]
as its number of follower neighbors, and let
\[
\underline{\deg}_f = \min_{j\in\mathcal{V}_f} \deg_f(j),
\]
denote the minimum follower–follower degree. 
If the following conditions hold:
\begin{itemize}
    \item[$i$)]  $\mathcal{G}$ is connected;
    \item[$ii$)]  Leaders are not directly connected, 
    i.e., $k \notin \mathcal{N}(j)$ for all $k,j\in\mathcal{V}_\ell$;
    \item[$iii$)] with $\phi(j)\!:=\!\frac{\deg(j)}{(\deg(j)\!+\!1\!-\!\lambda_F)}$, one has 
    \[1-\max_{j\in\mathcal{V}_\ell}\phi(j)
      > \max_{j\in\mathcal{V}_\ell}\min_{k\in\mathcal{V}_\ell}|\phi(j)-\phi(k)|;\]
    \item[$iv$)]  $\underline{\deg}_f$ is sufficiently large,
\end{itemize}
then the Fiedler vector $v_F$ of $\mathcal{G}$ satisfies
\[
\min_{i\in\mathcal{V}_f} [v_F]_i
-
\max_{i\in\mathcal{V}_\ell} [v_F]_i
\;>\;
\max_{i\in\mathcal{V}_\ell}\min_{j\in\mathcal{V}_\ell}
    \bigl|[v_F]_i-[v_F]_j\bigr|.
\]
\end{theorem}
%%%%%%%%%%%%%%%%%%%%
\begin{proof}
Define the positive margin
\begin{align}\label{epsilon_d}
\epsilon_d 
&= 1 
- \max_{j\in \mathcal{V}_\ell} 
  \phi(j)
- \max_{j\in \mathcal{V}_\ell} 
   \min_{k\in \mathcal{V}_\ell} 
   \Biggl|
  \phi(j)
   - \phi(k)
   \Biggr| > 0.
\end{align}

Consider a graph sequence 
$\{\mathcal{G}^{\sigma(i)}\}$ satisfying Definition~\ref{def:graph_sequence},
that is,
each $\mathcal{G}^{\sigma(i)}=(\mathcal{V}^{\sigma(i)},\mathcal{E}^{\sigma(i)})$
is connected, the leader set $\mathcal{V}_\ell$ and the leader degrees 
$\mathbf d_j$ are fixed,
leaders are not adjacent, and the minimum follower--follower degree
$\underline{\deg}_f^{\sigma(i)}$ increases monotonically with~$i$.
Let $v_\star^{\sigma(i)}$ denote the Fiedler vector associated with the graph $\mathcal G^{\sigma(i)}$.  Then, as a result of Lemma \ref{FiedlerDegree}, we have
$$\lim_{i\to \infty} \|\bar v_F^{\sigma(i)}-v_\star^{\sigma(i)}\| = 0,$$
where $\bar v_F^{\sigma(i)}$ is defined in \eqref{fiedler_vec_limit}.

Now, let $S^{\sigma(i)}=\|\bar v_F^{\sigma(i)}-v_\star^{\sigma(i)}\|$. Since $\displaystyle \lim_{i \to \infty} S^{\sigma(i)}=0$, for every $\epsilon>0$ there exists an index $i^*$ such that
$S^{\sigma(i)} < \epsilon$ for all $i>i^*$.
Choose $i^*$ so that $\underline{\deg}_f^{\sigma(i^*)}=m$
and fix $\epsilon < \epsilon_d/4$.
Then, for all $i>i^*$,
the vectors $\bar v_F^{\sigma(i)}$ and $v_*^{\sigma(i)}$
are $\epsilon$–close component-wise.

We now compare the follower and leader components of $v_\star^{\sigma(i)}$.  
Because $S^{\sigma(i)} < \epsilon$, each component of $v_\star^{\sigma(i)}$ differs from the corresponding entry of $\bar v_F^{\sigma(i)}$ by at most~$\epsilon$.  
Hence,
{\small{\begin{align*}
\min_{j\in\mathcal{V}_f}
   [v_\star^{\sigma(i)}]_j
 - \max_{j\in\mathcal{V}_\ell}
   [v_\star^{\sigma(i)}]_j
&>
\min_{j\in\mathcal{V}_f}
   [\bar v_F^{\sigma(i)}]_j
 - \max_{j\in\mathcal{V}_\ell}
   [\bar v_F^{\sigma(i)}]_j
 - 2\epsilon.
\end{align*}}}

Substituting the explicit definition of $\bar v_F^{\sigma(i)}$ from~\eqref{fiedler_vec_limit},
we have
{\small{\[
\min_{j\in\mathcal{V}_f} [\bar v_F^{\sigma(i)}]_j = 1, \text{ and }
\max_{j\in\mathcal{V}_\ell} [\bar v_F^{\sigma(i)}]_j
= \max_{j\in\mathcal{V}_\ell}
   \frac{\mathbf d_j}{\mathbf d_j+1-\lambda_F^{\sigma(i)}}.
\]}}
Therefore,
{\small{\begin{align}\label{eps_bound1}
\min_{j\in\mathcal{V}_f}
   [v_\star^{\sigma(i)}]_j
 - \max_{j\in\mathcal{V}_\ell}
   [v_\star^{\sigma(i)}]_j
>
1 -
\max_{j\in\mathcal{V}_\ell}
   \frac{\mathbf d_j}{\mathbf d_j + 1 - \lambda_F^{\sigma(i)}}
 - 2\epsilon.
\end{align}}}

% Because $S^{\sigma(i)} < \epsilon$, each component of $v_\star^{\sigma(i)}$ differs from the corresponding entry of $\bar v_F^{\sigma(i)}$ by at most $\epsilon$.  
% Hence,
% {\footnotesize{
% \begin{align*}
% \min_{j\in\mathcal{V}_f}
%    [v_\star^{\sigma(i)}]_j
%  - \max_{j\in\mathcal{V}_\ell}
%    [v_\star^{\sigma(i)}]_j> 
%    \min_{j\in\mathcal{V}_f}
%      [\bar v_F^{\sigma(i)}]_j
%    - \max_{j\in\mathcal{V}_\ell}
%      [\bar v_F^{\sigma(i)}]_j
%    - 2\epsilon^2.
% \end{align*}}}
% %\textcolor{blue}{above shouldn't be 2$\epsilon$ without the power?}
% Using the definition of $\bar v_F^{\sigma(i)}$, this becomes
% \[
% 1 - \max_{j\in\mathcal{V}_\ell}
%    \frac{\mathbf d_j}{\mathbf d_j+1-\lambda_F}
%    - 2\epsilon.
% \]

Recalling the definition of $\epsilon_d$ in \eqref{epsilon_d}, 
we can rearrange it as
\[
1 - 
\max_{j\in\mathcal{V}_\ell} \phi(j)
=
\max_{j\in\mathcal{V}_\ell}
  \min_{k\in\mathcal{V}_\ell}
  |\phi(j)-\phi(k)|
+ \epsilon_d.
\]
Substituting this identity into the  inequality \eqref{eps_bound1} gives
{\small{\begin{align*}
\min_{j\in\mathcal{V}_f}[v_\star^{\sigma(i)}]_j
-\max_{j\in\mathcal{V}_\ell}[v_\star^{\sigma(i)}]_j
>
\max_{j\in\mathcal{V}_\ell}\!
  \min_{k\in\mathcal{V}_\ell}\!
  |\phi(j)-\phi(k)|
+\epsilon_d-2\epsilon.
\end{align*}}}

% Introducing $\epsilon_d$ via its definition gives
% \begin{align*}
% &\max_{j\in\mathcal{V}_\ell}
%    \min_{k\in\mathcal{V}_\ell}
%    \Biggl|
%    \frac{\mathbf d_j}{\mathbf d_j+1-\lambda_F}
%    - \frac{\mathbf d_k}{\mathbf d_k+1-\lambda_F}
%    \Biggr|
%  + \epsilon_d - 2\epsilon.
% \end{align*}
Applying again the bound $S^{\sigma(i)}<\epsilon$ to replace the terms involving $\bar v_F^{\sigma(i)}$ by those of $v_\star^{\sigma(i)}$, we obtain
\begin{align*}
&\min_{j\in\mathcal{V}_f}
   [v_\star^{\sigma(i)}]_j
 - \max_{j\in\mathcal{V}_\ell}
   [v_\star^{\sigma(i)}]_j \\[1ex]
&\quad >
   \max_{j\in\mathcal{V}_\ell}
     \min_{k\in\mathcal{V}_\ell}
     \bigl|
     [v_\star^{\sigma(i)}]_j - [v_\star^{\sigma(i)}]_k
     \bigr|
   + \epsilon_d - 4\epsilon.
\end{align*}
Finally, since $\epsilon<\epsilon_d/4$, the right-hand side is strictly larger than 
$\max_{j\in\mathcal{V}_\ell}\min_{k\in\mathcal{V}_\ell}|[v_\star^{\sigma(i)}]_j - [v_\star^{\sigma(i)}]_k|$,  
and therefore
\[
\min_{j\in\mathcal{V}_f}[v_\star^{\sigma(i)}]_j
-
\max_{j\in\mathcal{V}_\ell}[v_\star^{\sigma(i)}]_j
>
\max_{j\in\mathcal{V}_\ell}\min_{k\in\mathcal{V}_\ell}
|[v_\star^{\sigma(i)}]_j-[v_\star^{\sigma(i)}]_k|.
\]

\smallskip
Thus, any graph $\mathcal{G}$ with $\underline{\deg}_f^{\mathcal{G}}>m$
satisfies the claimed inequality for its Fiedler vector $v_F$,
completing the proof.
\end{proof}

\section{A Leader Identification Algorithm}\label{sec:algorithm}

Section \ref{sec.leader_graphs} provides a characterization of graphs that have good separation between the entries of the Fiedler vector corresponding to the leaders and the entries corresponding to the followers.  Assuming that we are considering graphs with this property is not enough on its own to actually solve the problem outlined at the end of Section \ref{sec.statement}.  Indeed, what is missing is a way to estimate the Fiedler vector of the network using only data from the its trajectories generated by \eqref{semi_aut_orig}.

We next review the notion of \emph{relative tempo} introduced in~\cite{Shao2019relative}, which extends the earlier concept of the \emph{degree of relative influence} proposed in~\cite{Mesbahi2014}. 
These metrics quantify how the evolution rate of one agent (or subgroup) compares to another within a consensus-type network. 
Formally, the relative tempo between two agents is defined as
\begin{equation}
\tau_{ij} = \frac{\dot{x}_i}{\dot{x}_j},
\label{eq:tau_def}
\end{equation}
capturing the instantaneous ratio between their rates of change. 

For large times $t > T$, the agents’ dynamics are dominated by the slowest decaying mode of the network Laplacian.  
Projecting the system onto its eigenbasis and neglecting the zero and fast-decaying modes yields~\cite{Mesbahi2014}
\begin{equation}
\tau_{ij} \approx \frac{[v_F]_i}{[v_F]_j},
\label{eq:tau_approx}
\end{equation}
where $v_F$ is the Fiedler vector of $\mathcal G$. 
Consequently, the steady-state relative tempos encode the spatial structure of this dominant eigenvector. 
By selecting a reference agent $j^\ast$, one can estimate $v_F$ up to a scaling factor directly from relative tempo measurements:
\begin{equation}
v_F \in \mathrm{span}\{\bar{\tau}\}, \qquad [\bar{\tau}]_i = \tau_{i j^\ast}.
\label{eq:vf_estimation}
\end{equation}
This relationship provides a data-driven approach for recovering the Fiedler vector---and thus the influence structure of the network---from velocity measurements alone.

 We now propose an algorithm to identify leaders in semi-autonomous consensus networks whose graphs satisfy Definition \ref{def:graph_sequence} and Theorem \ref{thm:main}.
 %
 
% \vspace*{5pt}
\begin{algorithm}[H]
\caption{External Observer–Based Leader Identification}
\label{alg:IdentifyAlgo}
\begin{algorithmic}[1]
\State \textbf{Input:} Time series of agent states $x_i(t)$ under constant external input from \eqref{semi_aut_orig}.
\State \textbf{Output:} Estimated leader set $\mathcal{V}_\ell$.

\State Measure the agents' steady-state velocities $\dot{x}_i$.
\State Compute the relative tempo 
\[
\tau_{ij} = \frac{\dot{x}_i}{\dot{x}_j},
\]
and estimate the Fiedler vector $v_F$ using~\eqref{eq:vf_estimation}.
\State Sort the entries of $v_F$ in ascending order:
\[
v_{F_s} = \mathrm{sort}(v_F), \qquad [v_{F_s}]_i \le [v_{F_s}]_{i+1}.
\]
\State Determine the number of leaders:
\[
n_\ell = \arg\max_{j \in \{1,\dots,n-1\}} 
\big( [v_{F_s}]_{j+1} - [v_{F_s}]_j \big).
\]
\State Identify the leader nodes as those corresponding to the smallest $n_\ell$ components of $v_{F_s}$:
\[
\mathcal{V}_\ell = \{\, i \mid [v_F]_i \text{ among the smallest } n_\ell \,\}.
\]
\end{algorithmic}
\end{algorithm}

Algorithm \ref{alg:IdentifyAlgo} provides a systematic method to identify leaders based on the Fiedler vector derived from agents' velocity measurements. Line 3 ensures the system reaches steady state so that transient dynamics do not affect the estimation. Line 4 transforms the velocity data into the relative tempo, which gives us an estimation of the Fiedler vector. Sorting the Fiedler vector in line 5 reveals the natural separation between leaders and followers, and line 6 identifies the number of leaders by detecting the largest gap in the sorted vector. Finally, line 7 maps these components back to the actual leader agents.

\section{NUMERICAL EXAMPLE}\label{sec:NumericalExamples}

In this section, we provide examples of networks that support our findings from Theorem \ref{thm:main} and Algorithm \ref{alg:IdentifyAlgo}.
The following examples demonstrate a 2D scenario, where the leaders are the green nodes. Here, the dynamics of \eqref{semi_aut_orig} are simply augmented using a Kronecker product to represent agents in $\mathbb{R}^2$. We consider a system with $n=10$ agents, where $\{ 2,4,8 \}\in \mathcal{V}_\ell$. The external input provided to the leaders are given as
$$u_2^{ex} = [40,35]^T,\hspace{2mm} u_4^{ex} = [48,44]^T,\hspace{2mm}  u_8^{ex} =[16,45]^T.$$
%\todo[inline]{can't you make an example where the leaders are not 1,2,3? something more random so things look more impressive? \textcolor{blue}{changed, but need to consider the fact we assume the first $|\mathcal{V}_\ell|$ nodes are leaders, as defined in \ref{L11}}}

The underlying network graph is illustrated in Figure \ref{fig:graph_small1}. We simulate the protocol, and the resulting trajectories are shown in Figure \ref{fig:traj_small1}. The relative tempo is depicted in Figure \ref{fig:tempo_small1}. The plot shows with black-dashed lines the true values of the Fiedler vector.  As expected, the leader and follower components are clearly separated.

\begin{figure}[!h]
  \centering
\begin{tikzpicture}[scale=3]

% =========================
% Nodes (10) on a circle
% =========================
\foreach \i in {1,...,10}{
    \pgfmathsetmacro{\angle}{360/10 * (\i - 1)}
    \pgfmathsetmacro{\x}{cos(\angle)}
    \pgfmathsetmacro{\y}{sin(\angle)}
    % Green nodes: 2, 4, 8
    \ifnum\i=2
        \Vertex[size=.5, x=\x, y=\y, label=$\i$, color=green]{v\i}
    \else\ifnum\i=4
        \Vertex[size=.5, x=\x, y=\y, label=$\i$, color=green]{v\i}
    \else\ifnum\i=8
        \Vertex[size=.5, x=\x, y=\y, label=$\i$, color=green]{v\i}
    \else
        \Vertex[size=.5, x=\x, y=\y, label=$\i$]{v\i}
    \fi\fi\fi
}

% =========================
% Edges (new set)
% =========================
\Edge(v1)(v3)
\Edge(v1)(v6)
\Edge(v1)(v8)
\Edge(v1)(v9)
\Edge(v1)(v10)

\Edge(v2)(v3)
\Edge(v2)(v6)
\Edge(v2)(v9)

\Edge(v3)(v5)
\Edge(v3)(v6)
\Edge(v3)(v7)
\Edge(v3)(v8)
\Edge(v3)(v9)

\Edge(v4)(v5)
\Edge(v4)(v6)
\Edge(v4)(v10)

\Edge(v5)(v6)
\Edge(v5)(v7)
\Edge(v5)(v8)

\Edge(v6)(v7)
\Edge(v6)(v10)

\Edge(v7)(v9)
\Edge(v7)(v10)

\end{tikzpicture}
  \caption{Sensing graph of Example~1.}
  \label{fig:graph_small1}
\vspace{10mm}
  \includegraphics[width=0.9\linewidth]{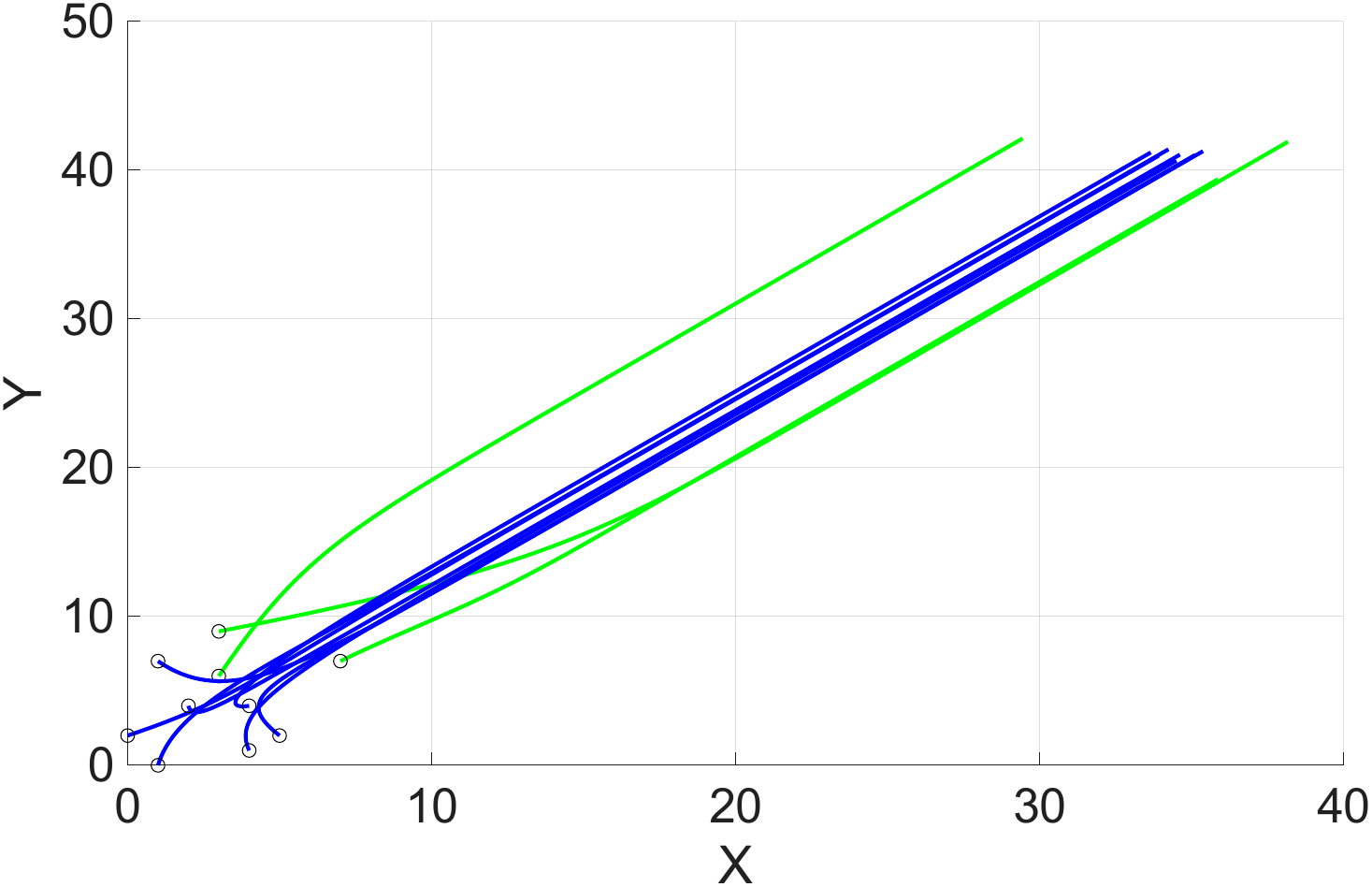}
  \caption{Trajectories corresponding to the underlying graph in Example~1. 
           Circles mark the initial states. The blue curve corresponds to the followers, while the green curve represents the leaders.}
  \label{fig:traj_small1}
\end{figure}

Next, we verify the conditions outlined in Theorem~\ref{thm:main}.
First, let us check condition~$(iii)$:
\begin{align*}
\epsilon_d 
&= 1 
- \max_{j\in \mathcal{V}_\ell} 
  \phi(j)
- \max_{j\in \mathcal{V}_\ell} 
   \min_{k\in \mathcal{V}_\ell} 
   \Biggl|
  \phi(j)
   - \phi(k)
   \Biggr|  \\
&= 0.2 > 0.
\end{align*}
We can see that $\epsilon_d$ has a positive value, which means condition ($iii$) is satisfied.
Next, for condition $(iv)$, we require a sufficiently large $\underline{\deg}_F$ to ensure $\epsilon<\frac{\epsilon_d}{4}$. We calculate:
$$\epsilon = 0.0459 < 0.05 = \frac{\epsilon_d}{4}.$$
Since all conditions are satisfied, there is sufficient separation between the Fiedler vector components corresponding to leaders and followers. Therefore, Algorithm~\ref{alg:IdentifyAlgo} can be used to identify the leaders.

Using the velocity measurements, we can estimate the Fiedler vector as 
\begin{align*}
v_F = [&\,0.3376,~0.2661,~0.3276,~0.2649,~0.3232,\\
       &~0.3277,~0.3501,~0.2638,~0.3420,~0.3417\,]^T.
\end{align*}
%Sorting the vector and then completing the steps of Algorithm \ref{alg:IdentifyAlgo} we find from the trajectory data that there are $3$ leaders. 
{Sorting the vector and then determining the number of leaders using the calculation in Step~6 of the algorithm reveals that there are three leaders. Finally, according to Step~7, the identified leaders are $\{ 2,4,8 \}\in \mathcal{V}_\ell$.
}
%\todo[inline]{please add black reference lines to the relative tempo figure to show the real fiedler vector values \textcolor{blue}{added}}
%\todo[inline]{remove titles from all figures (fig 5 and 6) \textcolor{blue}{removed}}
%\todo[inline]{fix captions to be more accurate - fig 5 has no diamond marks anywwhere, and they are not even important.  remove the legend from fig 5 \textcolor{blue}{fixed} but you removed the circles that show initial position - and that is useful to see.  The caption also has to be updated \textcolor{blue}{updated}}

% \begin{algorithmic}
% \State Step 1: Measure the agents' velocities to an external constant input until steady state.
% \State Step 2: Compute the relative tempo and the Fiedler vector (see Section~\ref{relativetempo}):

% \State Step 3: Sort the Fiedler vector, $v_{F_s} = \text{sort}(v_F)$, such that 
% $[v_{F_s}]_i \leq [v_{F_s}]_{i+1}$:
% \begin{align*}
% v_{F_s} = [&\,0.1798,~0.1804,~0.1817,~0.3276,~0.3288,\\
%            &~0.3311,~0.3631,~0.3805,~0.3831,~0.3920\,].
% \end{align*}

% \State Step 4: Calculate the number of leaders $n_l$ with 
% $$n_l=|\mathcal{V}_\ell| = \arg\max_{j\in \{1,2,3,\cdots,n-1\}} \{[v_{F_s}]_{j+1} - [v_{F_s}]_{j}\}.$$
% $$n_l = 3$$
% \State Step 5: The leaders correspond to the smallest $n_l$ components in $v_{F_s}$.
% $$\mathcal{V}_\ell = \{1,2,3\}.$$
% \end{algorithmic}
% We successfully identified the leader agents.

\begin{figure}[!t]
  \centering
  \includegraphics[width=0.9\linewidth]{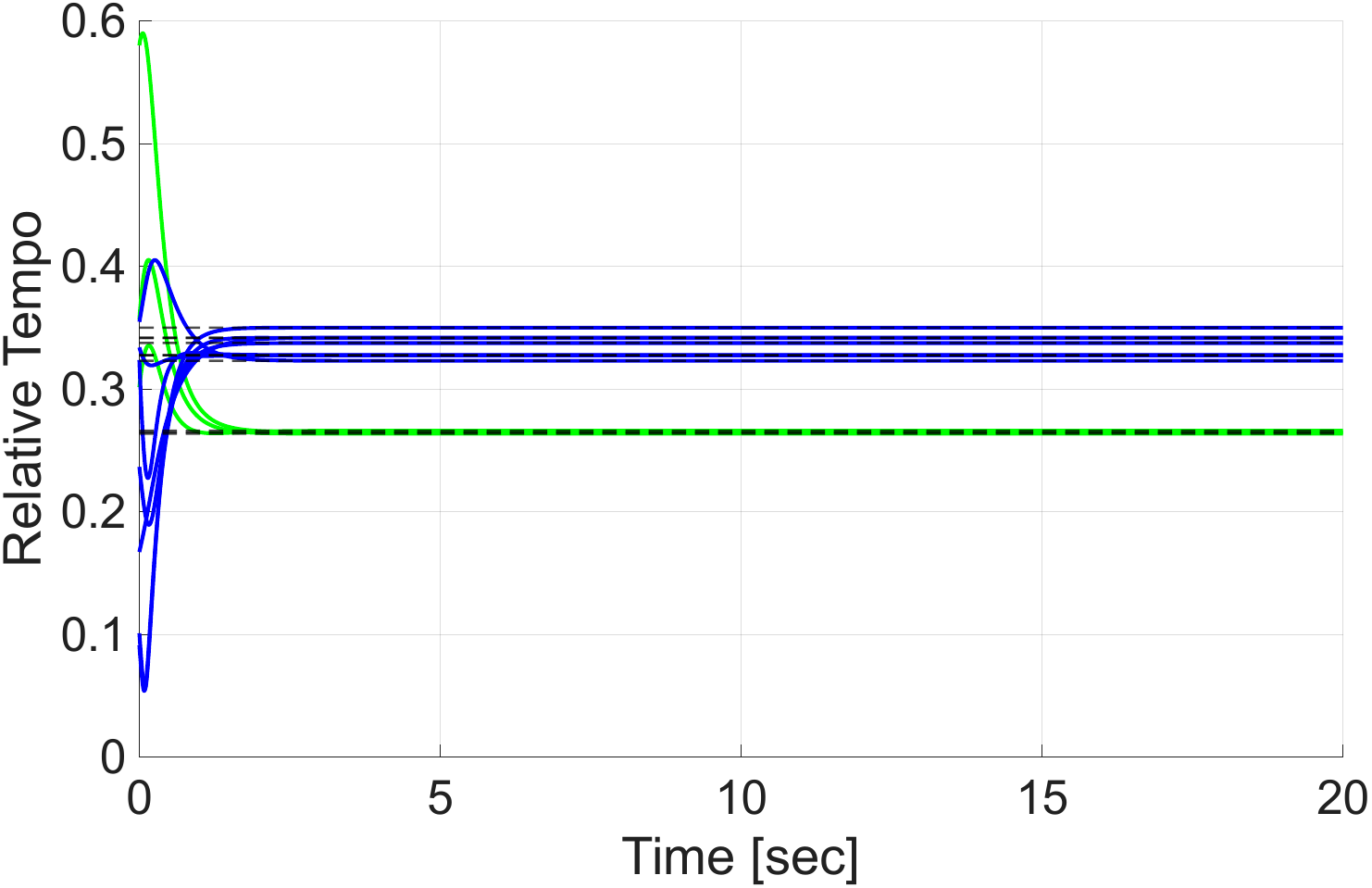}
  \caption{Relative tempo for the network in Example~1. 
           Note the gap between leaders and follower components. The blue curve corresponds to the followers, while the green curve represents the leaders. The black dashed lines is the true value of the Fiedler vector.}
  \label{fig:tempo_small1}
\end{figure}

\section{CONCLUSIONS}\label{sec.conclusion}

Our findings indicate that certain graph structures are more likely to exhibit a clear separation between the components of the Fiedler vector. Specifically, the graphs identified in this research tend to have relatively dense connectivity—characterized by a high mean node degree—while the leader nodes themselves maintain comparatively lower degrees.
This structural property enables effective leader identification through external observation, particularly in scenarios involving constant external input. In such cases, the agents’ steady-state velocities are directly related to the corresponding components of the Fiedler vector. This relationship allows us to reliably distinguish and identify the leader agents within the network.

Future research can extend this work in several directions. One promising direction is to investigate scenarios involving non-constant external input signals, or inputs that remain constant only over specific time intervals—corresponding to signals with discrete frequency updates. Another important avenue is the development of methods for identifying not only the leader nodes but the complete underlying network structure. Finally, exploring a broader range of graph topologies that exhibit component separation in the Fiedler vector could provide deeper insight into how structural properties influence leader–follower dynamics.

\bibliographystyle{IEEEtran}
\bibliography{references}

\end{document}